\documentclass[preprint,12pt]{elsarticle}
\usepackage{amsfonts}

 \usepackage{graphics}
\usepackage{graphicx}
 \usepackage{epsfig}
\usepackage{amssymb}
\usepackage{amsthm}

\newtheorem{thm}{Theorem}
\newtheorem{lem}[thm]{Lemma}
\newdefinition{rmk}{Remark}
\newproof{pf}{Proof}
\newdefinition{example}{Example}
\newdefinition{definition}{Definition}
\newdefinition{proposition}{Proposition}
\newdefinition{corollary}{Corollary}
\newdefinition{conjecture}{Conjecture}
\newdefinition{problem}{Problem}

\begin{document}

\begin{frontmatter}



\title{Double $B$-tensors and quasi-double $B$-tensors}

\author[]{Chaoqian Li}
\author[]{Yaotang Li\corref{cor1}}
\ead{liyaotang@ynu.edu.cn}

\cortext[cor1]{Corresponding author.Yaotang Li, School of
Mathematics and Statistics, Yunnan University, Kunming, Yunnan,
650091, PR China. Email: liyaotang@ynu.edu.cn. This author's work
was supported by the National Natural Science Foundation of P.R.
China (Grant No. 11361074). His work was partially done when he was
visiting The Hong Kong Polytechnic University.}

\cortext[]{Chaoqian Li: School of Mathematics and Statistics, Yunnan
University, Kunming, Yunnan, 650091, PR China. Email:
lichaoqian@ynu.edu.cn. This author's work was supported by the
National Natural Science Foundation of P.R. China (Grant No.
11326242) and Natural Science Foundations of Yunnan Province (Grant
No. 2013FD002).}

\address[]{School of Mathematics and Statistics, Yunnan University,
Kunming, Yunnan,  P. R. China 650091}

\begin{abstract}
In this paper, we propose two new classes of tensors: double
$B$-tensors and quasi-double $B$-tensors, give some properties of
double $B$-tensors and quasi-double $B$-tensors, discuss their
relationships with $B$-tensors and positive definite tensors and
show that even order symmetric double $B$-tensors and even order
symmetric quasi-double $B$-tensors are positive definite. These give
some checkable sufficient conditions for positive definiteness of
tensors.
\end{abstract}


\end{frontmatter}


\section{Introduction} A real
order $m$ dimension $n$ tensor $\mathcal{A}=(a_{i_1\cdots i_m})$
consists of $n^m$ real entries:
\[a_{i_1\cdots i_m}\in \mathbb{R},\] where $i_j\in N=\{1,2,\ldots,n\}$ for
$j=1,\ldots, m$  \cite{Ch1,Di,He,Li,Qi}. It is obvious that a matrix
is an order 2 tensor. Moreover, a tensor $\mathcal {A}=(a_{i_1\cdots
i_m})$ is called symmetric \cite{Li1,Qi} if
\[a_{i_1\cdots i_m }= a_{\pi(i_1\cdots i_m )},\forall \pi\in
\Pi_m,\]where $\Pi_m$ is the permutation group of $m$ indices. And
an order $m$ dimension $n$ tensor is called the unit tensor denoted
by $\mathcal{I} $ \cite{Ch1,Ya}, if its entries are
$\delta_{i_1\cdots i_m}$ for $i_1,\ldots ,i_m \in N$, where
\[\delta_{i_1\cdots i_m}=\left\{\begin{array}{cc}
   1,   &if~ i_1=\cdots =i_m,  \\
   0,   &otherwise.
\end{array}
\right.\] For a tensor  $\mathcal{A}$ of order $m$ dimension $n$, if
there is a nonzero vector $ x=(x_1,\ldots,x_n)^T\in \mathbb{R}^{n}$
and a number $\lambda \in \mathbb{R}$ such that
\[ \mathcal{A}x^{m-1}=\lambda x^{[m-1]},\]
where \[(\mathcal{A}x^{m-1})_i=\sum\limits_{i_2,\ldots,i_m \in N}
a_{ii_2\cdots i_m}x_{i_2}\cdots x_{i_m}\] and $x^{[m-1]}=(x_1^{m-1},\ldots,x_n^{m-1})^T,$
then $\lambda$ is called an H-eigenvalue of $\mathcal{A}$ and $x$ is called an H-eigenvector of $\mathcal{A}$ \cite{Qi}.

As a natural extension of $B$-matrices \cite{Pe,Pe1}, $B$-tensors is introduced by Song and Qi \cite{So1}.

\begin{definition} \cite{So1} \label{def1.1} Let $\mathcal{B}=(b_{i_1\cdots i_m})$ be a real tensor of order $m$ dimension
$n$. $\mathcal{B}$ is called a $B$-tensor if for all $i\in N$
\[\sum \limits_{i_2,\ldots,i_m \in N} b_{ii_2\cdots i_m} >0\]
and
\[\frac{1}{n^{m-1}} \left(\sum \limits_{i_2,\ldots,i_m \in N} b_{ii_2\cdots i_m} \right) >b_{ij_2\cdots j_m},~for~j_2,\ldots, j_m\in N, \delta_{ij_2\cdots j_m}=0. \]
\end{definition}

By Definition \ref{def1.1}, Song and Qi \cite{So1} gave the
following property of $B$-tensors.

\begin{proposition} \cite[Proposition 3]{So1}\label{pro1.1}
Let $\mathcal{B}=(b_{i_1\cdots i_m})$ be a real tensor of order $m$ dimension $n$.
Then  $\mathcal{B}$ is a $B$-tensor if and only if for each $i\in N$,
\begin{equation} \label{eq3.1}\sum\limits_{i_2,\ldots,i_m\in N} b_{i_1\cdots i_m} > n^{m-1} \beta_i(\mathcal{B}),\end{equation}
where \[\beta_i(\mathcal{B})=\max\limits_{j_2,\ldots,j_m\in N,\atop \delta_{ij_2\ldots j_m}=0} \{0, b_{ij_2\cdots j_m}\}.\]
\end{proposition}

It is easy to see that Inequality (\ref{eq3.1}) is equivalent to
\begin{equation} \label{eq3.2} b_{ii\cdots i}-\beta_i(\mathcal{B}) > \Delta_i ({\mathcal{B}}), \end{equation}
where
\begin{equation} \label{eq1.2}\Delta_i ({\mathcal{B}})= \sum\limits_{i_2\ldots i_m \in N, \atop \delta_{ii_2\ldots i_m}=0}
 (\beta_i(\mathcal{B})-b_{ii_2\cdots i_m}). \end{equation}Hence, we by Inequality (\ref{eq3.2}) obtain another property for $B$-tensors.

\begin{proposition} \label{pro1.2}
Let $\mathcal{B}=(b_{i_1\cdots i_m})$ be a real tensor of order $m$ dimension $n$.
Then  $\mathcal{B}$ is a $B$-tensor if and only if for each $i\in N$, Inequality (\ref{eq3.2}) holds.
\end{proposition}

$B$-tensors are linked with positive definite tensors and $M$-tensors, which are useful in automatical control, magnetic resonance imaging and spectral hypergraph theory \cite{Bo,Ch,Di,Ha,He,Hu,Hu1,Hu2,Hu3,Hu4,Qi3,Qi4,Qi5,Qi6,So1,Wa,Wa1,Yu}.

\begin{definition} \label{def1.2}\cite{Qi7,So1} Let $\mathcal{A}=(a_{i_1\cdots i_m})$ be a real tensor of order $m$ dimension $n$. $\mathcal{A}$ is called positive definite if for any nonzero vector $x$ in $\mathbb{R}^{n}$,
\[\mathcal{A}x^{m} >0,\]
and positive semi-definite if for any vector $x$ in $\mathbb{R}^{n}$,
\[ \mathcal{A}x^{m} \geq 0,\]
where $\mathcal{A}x^{m}=\sum\limits_{i_1,i_2,\ldots,i_m \in N}
a_{i_1i_2\cdots i_m}x_{i_1}\cdots x_{i_m}$.
\end{definition}

One of the most important properties of $B$-tensors is listed as follows.

\begin{thm} \label{th1.3}\cite{Qi7} Let $\mathcal{B}=(a_{i_1\cdots i_m})$ be a real tensor of order $m$ dimension $n$. If $\mathcal{B}$ is an even order symmetric $B$-tensor, then $\mathcal{B}$ is positive definite.
\end{thm}

The definition of $DB$-matrix is a generalization of the $B$-matrix
\cite{Pe1}. Here we call a matrix $A=(a_{ij})\in \mathbb{R}^{n\times
n}$ a $DB$-matrix if for any $i,j \in N$, $i\neq j$,
\[ \left( a_{ii}-\beta_i(A)\right)\left( a_{jj}-\beta_j(A)\right) > \sum\limits_{k\neq i} \left(\beta_i(A) -a_{ik}\right)\sum\limits_{k\neq j} \left(\beta_j(A) -a_{jk}\right).\]
A natural question is that can we extend the class of $DB$-matrices
to tensors with order $m\geq 3$ such that it has the property like
that in Theorem \ref{th1.3}, that is, whether or not an even order
symmetric  tensor $\mathcal{A}=(a_{i_1\ldots i_m})$ satisfying
\begin{equation}\label{eq03.1}
(a_{i\cdots i}-\beta_i(\mathcal{A}))(a_{j\cdots j}-\beta_j(\mathcal{A}))
>\Delta_i ({\mathcal{A}})\Delta_j ({\mathcal{A}}),
\end{equation}
is positive definite? We see an example firstly for discussing the
question.

Consider the symmetric tensor $\mathcal {A}=(a_{ijkl})$  of order 4
dimension 2 defined as follows:
\[a_{1111}=a_{2222}=2, a_{1222}= a_{2122}= a_{2212}= a_{2221}=-1,\]
and other $a_{ijkl}=0$.
By calculation, we have $ \beta_1(\mathcal {A})=\beta_2(\mathcal {A})=0$, and
\[a_{1111}a_{2222}=4>3=\sum\limits_{\delta_{1jkl}= 0} (-a_{1jkl}) \sum\limits_{\delta_{2jkl}= 0} (-a_{2jkl}),\]
which satisfies Inequality (\ref{eq03.1}). However, $\mathcal{A}$ is not positive definite. In fact, for any entrywise positive vector
$x=(x_1,x_2)^T$. If $\mathcal {A} x^{4}>0$, that is,
\[\left\{\begin{array}{ccc}
   a_{1111}x_1^4 +\sum\limits_{j,k,l\in \{1,2\},\atop
\delta_{1jkl}=0} a_{1jkl}x_1x_{j}
x_{k}x_{l} &>&0,    \\
  a_{2222}x_2^4 +\sum\limits_{j,k,l\in \{1,2\},\atop
\delta_{2jkl}=0} a_{2jkl}x_2x_{j} x_{k}x_{l} &>&0,
\end{array}
\right.\] equivalently,
\[\left\{\begin{array}{ccc}
   2x_1^4 -x_1x_2^3 &>&0,   \\
  2x_2^4 - 3x_1x_2^3&>&0,
\end{array}
\right.\] then
\[\left\{\begin{array}{ccc}
   2x_1^3  &>&x_2^3,   \\
  x_2  &>&\frac{3}{2}x_1,
\end{array}
\right.\] which implies \[2x_1^3  >x_2^3>\frac{27}{8}x_1^3.\] This is
a contradiction. Hence, for any vector
$x=(x_1,x_2)^T>0 $, $\mathcal {A} x^{4}>0$ dosen't hold. Hence
$\mathcal{A}$ is not positive definite by Definition \ref{def1.2}.

The example shows that Inequality (\ref{eq03.1}) doesn't guarantee
the positive definiteness of tensor $\mathcal{A}$. In this paper, we
introduce two new classes of tensors by adding other conditions:
double $B$-tensors and quasi-double $B$-tensors as generalizations
of $B$-tensors, and prove that an even order symmetric
(quasi-)double $B$-tensor is  positive definite.

\section{Double $B$-tensor and quasi-double $B$-tensor}
Now, we present the definitions of double $B$-tensors and
quasi-double $B$-tensors.

\begin{definition}\label{def03.0}
Let $\mathcal{B}=(b_{i_1\cdots i_m})$ be a real tensor of order $m$ dimension $n$ with $b_{i\cdots i} > \beta_i(\mathcal{B})$ for all $i\in N$.
$\mathcal{B}$ is called a double $B$-tensor if $\mathcal{B}$ satisfies

(I) for any $i\in N$,
\[ b_{i\cdots i}-\beta_i(\mathcal{B}) \geq \Delta_i ({\mathcal{B}}),\]
where $\Delta_i ({\mathcal{B}})$ is defined as (\ref{eq1.2});

(II) for all $i,j\in N$, $i\neq j$, Inequality (\ref{eq03.1}) holds.
\end{definition}

\begin{definition}\label{def03.1}
Let $\mathcal{B}=(b_{i_1\cdots i_m})$ be a real tensor of order $m$ dimension $n$ with $b_{i\cdots i} > \beta_i(\mathcal{B})$ for all $i\in N$.
$\mathcal{B}$ is called a quasi-double $B$-tensor if for all $i,j\in N$, $i\neq j$,
\begin{equation} \label{eq3.3}
(b_{i\cdots i}-\beta_i(\mathcal{B})) \left (b_{j\cdots j}-\beta_j(\mathcal{B})- \Delta_j^i ({\mathcal{B}}) \right)
>\left( \beta_j(\mathcal{B})-b_{ji\cdots i}\right)\Delta_i ({\mathcal{B}}),
\end{equation}
where
\[\Delta_j^i ({\mathcal{B}}) = \Delta_j ({\mathcal{B}})- \left( \beta_j(\mathcal{B})-b_{ji\cdots i}\right)=\sum\limits_{\delta_{jj_2\ldots j_m}=0, \atop \delta_{ij_2\ldots j_m}=0} (\beta_j(\mathcal{B})-b_{jj_2\cdots j_m}).\]
\end{definition}

We now give some properties of double $B$-tensors and quasi-double $B$-tensors.

\begin{proposition} \label{pro2.1}
Let $\mathcal{B}=(b_{i_1\cdots i_m})$ be a real tensor of order $m$ dimension $n$.
(I) If $\mathcal{B}$ is a double $B$-tensor, then there is at most one $i\in N$ such that
\[ b_{i\cdots i}-\beta_i(\mathcal{B}) = \Delta_i ({\mathcal{B}}).\]
(II) If $\mathcal{B}$ is a quasi-double double $B$-tensor,
  then there is at most one $i\in N$ such that
\[ b_{i\cdots i}-\beta_i(\mathcal{B}) \leq \Delta_i ({\mathcal{B}}).\]
\end{proposition}

\begin{proof} We only prove that (II) holds, and (I) is proved similarly.  Suppose that there are $i_0$ and $j_0$ such that
\[ b_{i_0\cdots i_0}-\beta_{i_0}(\mathcal{B}) \leq \Delta_{i_0} ({\mathcal{B}}),\]
 and
\[ b_{j_0\cdots j_0}-\beta_{j_0}(\mathcal{B}) \leq \Delta_{j_0} ({\mathcal{B}}),\] equivalently
\[ b_{j_0\cdots j_0}-\beta_{j_0}(\mathcal{B}) -\Delta_{j_0}^{i_0} ({\mathcal{B}}) \leq \beta_{j_0}(\mathcal{B})-b_{j_0i_0\cdots i_0},\]  If $b_{j_0\cdots j_0}-\beta_{j_0}(\mathcal{B}) -\Delta_{j_0}^{i_0} ({\mathcal{B}})<0$, then
\begin{equation}\label{eq2.1} (b_{i_0\cdots i_0}-\beta_{i_0}(\mathcal{B})) \left(b_{j_0\cdots j_0}-\beta_{j_0}(\mathcal{B}) -\Delta_{j_0}^{i_0} ({\mathcal{B}}) \right)\leq (\beta_{j_0}(\mathcal{B})-b_{j_0i_0\cdots i_0})\Delta_{i_0} ({\mathcal{B}}) ,\end{equation}
otherwise, $b_{j_0\cdots j_0}-\beta_{j_0}(\mathcal{B}) -\Delta_{j_0}^{i_0} ({\mathcal{B}})\geq 0$, which also leads to  Inequality (\ref{eq2.1}). This contradicts to the definition of quasi-double double $B$-tensors. The conclusion follows.
\end{proof}

The relationships of $B$-tensors, double $B$-tensors and quasi-double $B$-tensors are given as follows.

\begin{proposition} \label{pro2.2}
Let $\mathcal{B}=(b_{i_1\cdots i_m})$ be a tensor of order $m$ dimension $n$.
If $\mathcal{B}$ is a $B$-tensor, then $\mathcal{B}$ is a double $B$-tensor and a quasi-double $B$-tensor. Furthermore, if $\mathcal{B}$ is a double $B$-tensor,
 then $\mathcal{B}$ is  a quasi-double $B$-tensor.
\end{proposition}

\begin{proof} If $\mathcal{B}$ is a $B$-tensor, then by Proposition \ref{pro1.2} for any $i\in N$,
\[ b_{i\cdots i}-\beta_i(\mathcal{B}) > \Delta_i ({\mathcal{B}}),\]
that is,
\[b_{i\cdots i}-\beta_i(\mathcal{B})- \Delta_i^k ({\mathcal{B}})
>\beta_i(\mathcal{B})-b_{ik\cdots k}, ~for~ k\neq i.\]
Obviously, Inequality (\ref{eq03.1}) holds for any $i\neq j$. This implies that $\mathcal{B}$ is a double $B$-tensor. On the other hand, note that
for $i,j\in N$, $j\neq i$,
\[ b_{i\cdots i}-\beta_i(\mathcal{B}) > \Delta_i ({\mathcal{B}}),\]
and
\[b_{j\cdots j}-\beta_j(\mathcal{B})- \Delta_j^i ({\mathcal{B}})
>\beta_j(\mathcal{B})-b_{ji\cdots i}.\]
It is easy to see that Inequality (\ref{eq3.3}) holds, i.e.,  $\mathcal{B}$ is a quasi-double $B$-tensor by  Definition \ref{def03.1}.

Furthermore, if $\mathcal{B}$ is a double $B$-tensor, then there is at most one $k\in N$ such that $ b_{i\cdots i}-\beta_i(\mathcal{B}) = \Delta_i ({\mathcal{B}})$. If there is not $k\in N$ such that $b_{k\cdots k}-\beta_k(\mathcal{B}) = \Delta_k ({\mathcal{B}})$, then $\mathcal{B}$ is a $B$-tensor, consequently, $\mathcal{B}$ is a quasi-double $B$-tensor. If there is only one $k\in N$ such that $b_{k\cdots k}-\beta_k(\mathcal{B}) = \Delta_k ({\mathcal{B}})$, then we have for any $j\neq k$ \[b_{j\cdots j}-\beta_j(\mathcal{B}) > \Delta_j ({\mathcal{B}}).\] Note that for any $j \in N $, $j\neq k$,
\[b_{k\cdots k}-\beta_k(\mathcal{B}) = \Delta_k ({\mathcal{B}}), ~b_{k\cdots k}-\beta_k(\mathcal{B})- \Delta_k^j ({\mathcal{B}})=\beta_k(\mathcal{B})-b_{kj\cdots j},\]
and
\[ b_{j\cdots j}-\beta_j(\mathcal{B}) > \Delta_j ({\mathcal{B}}),~ b_{j\cdots j}-\beta_j(\mathcal{B})- \Delta_j^k ({\mathcal{B}})=\beta_j(\mathcal{B})-b_{jk\cdots k}.\]
This implies that Inequality (\ref{eq3.3}) holds, i.e., $\mathcal{B}$ is a quasi-double $B$-tensor.
\end{proof}

\begin{rmk}  (I) It is not difficult to see that the class of $B$-tensors is a proper subclass of double $B$-tensors and quasi-double  $B$-tensors, that is,
 \[\{B-tensors\} \subset \{double~ B-tensors\} \]
 and
 \[\{B-tensors\} \subset \{quasi-double~ B-tensors\}.\]

 (II) The class of double $B$-tensors is a proper subclass of quasi-double $B$-tensors. Consider the  tensor $\mathcal {A}=(a_{ijk})$
of order 3 dimension 2 defined as follows:
\[\mathcal {A}=[A(1,:,:),A(2,:,:)],\]
where
\[A(1,:,:)=\left(\begin{array}{cccc}
   2   &0   \\
   0   &-0.3
\end{array}
\right),~ A(2,:,:)=\left(\begin{array}{cccc}
   -1   &-0.3         \\
   -1.5   &2
\end{array}
\right).\]
By calculation,
 $\beta_1(\mathcal {A})=\beta_2(\mathcal {A})=0$, and
$a_{222}=2<2.8= \sum\limits_{\delta_{2jk}= 0} (-a_{2jk}).$
Hence $A$ is not a double $B$ tensor. Since
\[a_{111} \left (a_{222}- \Delta_2^1 ({\mathcal{A}}) \right)=0.4
>0.3 =\left(-b_{211}\right)\Delta_1 ({\mathcal{B}})\]
and
\[ a_{222}\left(a_{111}-\Delta_1^2(\mathcal {A})\right)=4>0.84=(-a_{122})\Delta_2(\mathcal {A}),\]
then $\mathcal {A}$ is a quasi-double $B$-tensor. Hence, the class of  double $B$-tensors is a proper subclass of quasi-double $B$-tensors. By (I), we have
 \[\{B-tensors\} \subset \{double~ B-tensors\} \subset  \{quasi-double~ B-tensors\}.\]
\end{rmk}

As is well known,  a $B$-matrix is a $P$-matrix \cite{Pe1,Pe1}. This
is not true for higher order tensors, that is, a $B$-tensor may not
be a $P$-tensor. In \cite{So1}, Song and Qi proved that a symmetric
tensor is a $P$-tensor if and only it is positive definite.
\begin{definition} \cite{So1}
A real tensor $\mathcal{A}=(a_{i_1\cdots i_m})$ of order $m$ dimension $n$ is called a $P$-tensor if for any nonzero $x$ in $\mathbb{R}^{n}$,
\[\max\limits_{i\in N} x_i (\mathcal{A}x^{m-1})_i >0.\]
\end{definition}

It is  pointed out in \cite{Qi7,So1} that an odd order $B$-tensor
may not be a $P$-tensor, Furthermore, Yuna and You \cite{Yu} gave an
example to show that an even order nonsymmetric $B$-tensor may not
be a $P$-tensor. Hence, by Proposition \ref{pro2.2}, we conclude
that an odd order (quasi-)double $B$-tensor may not be a $P$-tensor,
and an even order nonsymmetric (quasi-)double $B$-tensor may not be
a $P$-tensor. On the other hand, it is pointed out in \cite{Qi7}
that an even order symmetric $B$-tensor is a $P$-tensor. A natural
question is that whether or not an even order symmetric
(quasi-)double $B$-tensor is a $P$-tensor? In the following section,
we will  answer this question by discussing the positive
definiteness of (quasi-)double $B$-tensors.

\section{Positive definiteness}
Now, we discuss the positive definiteness of (quasi-)double
B-tensors. Before that some definitions are given.

\begin{definition} \cite{Di,He,Zh} Let $\mathcal{A}=(a_{i_1\cdots i_m})$ be a real tensor of order $m$ dimension $n$. $\mathcal{A}$ is called a $Z$-tensor if all of the off-diagonal entries of $\mathcal{A}$ are non-positive;

\end{definition}

\begin{definition} \label{def2.0} Let $\mathcal{A}=(a_{i_1\cdots i_m})$ be a tensor of order $m$ dimension $n\geq 2$.  $\mathcal{A}$ is called a doubly strictly diagonally dominant tensor (DSDD) if

(I) when $m=2$, $\mathcal{A}$ satisfies
\begin{equation} \label{eq02.1} |a_{i\cdots i}||a_{j\cdots j}|>
r_i(\mathcal{A})r_j(\mathcal{A}), ~for ~any ~i,j\in N, ~i\neq j, \end{equation}

(II) when  $m>2$, $\mathcal{A}$ satisfies  $|a_{i\cdots i}|\geq r_i(\mathcal{A})$ for any $i\in N$ and Inequality (\ref{eq02.1}) holds.
\end{definition}

Note here that when $m>2$, the similar condition that $|a_{i\cdots
i}|\geq r_i(\mathcal{A})$ for any $i\in N$, is necessary for DSDD
tensors to have the properties of  doubly strictly diagonally
dominant matrices; for details, see \cite{Li,Li1}.

\begin{definition} \label{def2.1} Let $\mathcal{A}= (a_{i_1\cdots i_m})$ be a tensor of order $m$ dimension $n\geq 2$.  $\mathcal{A}$ is called a quasi-doubly strictly diagonally dominant tensor (Q-DSDD) if for $i,j\in N$, $j\neq i$,
\begin{equation}\label{equ3.8}|a_{i\cdots
i}|\left(|a_{j\cdots j}|-r_j^i(\mathcal{A})\right)>
r_i(\mathcal{A})|a_{ji\cdots i}|,\end{equation}
where \[r_j^i(\mathcal
{A})=\sum\limits_{j_2,\ldots,j_m\in N,\atop \delta_{jj_2\ldots
j_m}=0} |a_{jj_2\cdots j_m}|=\sum\limits_{\delta_{jj_2\ldots j_m}=0,\atop
\delta_{ij_2\ldots j_m}=0} |a_{jj_2\cdots j_m}|-|a_{ji\cdots
i}|=r_j(\mathcal {A})-|a_{ji\cdots i}|.\]
\end{definition}

The relationships between (Q-)DSDD tensors and (quasi-)double $B$-tensors are established as follows.

\begin{proposition} \label{pro3.3}
Let $\mathcal{B}=(b_{i_1\cdots i_m})$ be a $Z$-tensor of order $m$ dimension $n$.
Then

(I) $\mathcal{B}$ is a double $B$-tensor if and only if $\mathcal{B}$ is a DSDD
tensor.

(II) $\mathcal{B}$ is a quasi-double $B$-tensor if and only if $\mathcal{B}$ is a Q-DSDD
tensor.
\end{proposition}

\begin{proof} We only prove that (II) holds, (I) can be obtained similarly. Since $\mathcal{B}$ be a $Z$-tensor, all of its off-diagonal entries are non-positive. Thus,  we have that for any $i\in N$,  $\beta_i(\mathcal{B})=0$,
\[ |b_{ii_2\cdots i_m}|=-b_{ii_2\cdots i_m}, ~for~all~i_2,\ldots,i_m \in N, \delta_{ii_2\ldots i_m}=0,\]
\[ r_i(\mathcal{B})=\Delta_i(\mathcal{B}) =\sum\limits_{i_2\ldots i_m \in N, \atop \delta_{ii_2\ldots i_m}=0} (\beta_i(\mathcal{B})-b_{ii_2\cdots i_m}),\]
and
\[ r_j^i(\mathcal{B})=\Delta_j^i(\mathcal{B}) =\sum\limits_{\delta_{jj_2\ldots j_m}=0, \atop \delta_{ij_2\ldots j_m}=0} (\beta_j(\mathcal{B})-b_{jj_2\cdots j_m}), ~for~j\neq i.\]
which implies that
Inequality (\ref{eq3.3}) is equivalent to  Inequality (\ref{equ3.8}). The conclusion follows. \end{proof}

In \cite{Li}, Li et al. gave some sufficient conditions for positive
definiteness of tensors.

\begin{lem} \cite[Theorem 11]{Li}\label{lem30.1} Let $\mathcal{A}=(a_{i_1\cdots i_m})$ be an even order real symmetric tensor of order $m$ dimension $n>2$ with $a_{k\cdots k}> 0$ for all $k\in  N$.
If $\mathcal{A}$ satisfies the condition (II) in Definition \ref{def2.0}, then $\mathcal{A}$ is positive definite.
\end{lem}

\begin{lem}\cite[Theorem 13]{Li} \label{lem30.2} Let $\mathcal{A}=(a_{i_1\cdots i_m})$ be an even order real symmetric tensor of order $m$ dimension $n>2$ with $a_{k\cdots k}> 0$ for all $k\in  N$.
If there is an index $i\in N$ such that for all $j\in N $, $j\neq i$, such that Inequality (\ref{equ3.8}) holds and $|a_{i\cdots i}|\geq r_i(\mathcal{A})$, then $\mathcal{A}$ is positive definite.
\end{lem}

By Lemmas \ref{lem30.1} and \ref{lem30.2}, we can easily obtain the following result.

\begin{thm} \label{imp}
An even order real symmetric (Q-)DSDD tensor is positive definite.
\end{thm}

Now according to Theorem \ref{imp}, we research the positive definiteness of symmetric (quasi-)double $B$-tensors. Before that we give the definition of partially all one tensors, which proposed by Qi and Song \cite{Qi7}. Suppose that $\mathcal{A}$ is a symmetric tensor of order $m$ dimension $n$, and has a principal sub-tensor $\mathcal{A}_r^J$ with $J\in N$ and $|J|=r (1\leq r \leq n)$ such that all the entries of $\mathcal{A}_r^J$are one, and all the other entries of $\mathcal{A}$ are zero, then $\mathcal{A}$ is called a partially all one tensor, and denoted by $\varepsilon^{J}$. If $J=N$, then we denote $\varepsilon^{J}$ simply by $\varepsilon$ and call it an all one tensor.  And an even order partially all one tensor is positive semi-definite; for details, see \cite{Qi7}.

\begin{thm} \label{th3.2}
Let $\mathcal{B}=(b_{i_1\cdots i_m})$ be a symmetric quasi-double $B$-tensor of order $m$ dimension $n$. Then either $\mathcal{B}$ is a Q-DSDD symmetric $Z$-tensor itself, or we have
\begin{equation}\label{eq300.4}
\mathcal{B}=\mathcal{M}+ \sum\limits_{k=1}^{s}h_k\varepsilon^{\hat{J}_k},
\end {equation}
where $\mathcal{M}$ is a Q-DSDD symmetric $Z$-tensor, $s$ is a positive integer, $h_k >0$ and $\hat{J}_k\subseteq N$, for $k=1,2,\cdots,s$. Furthermore, If $m$ is even, then $\mathcal{B}$ is
 positive definite, consequently, $\mathcal{B}$ is a $P$-tensor.
\end{thm}

\begin{proof} Let
$\hat{J}(\mathcal{B}) =\{i\in N:$~there~is~at~least~one positive~off-diagonal~entry~in~the \emph{i}th~row~of $\mathcal{B}\}.$ Obviously, $\hat{J}(\mathcal{B}) \subseteq N$. If $\hat{J}(\mathcal{B})=\emptyset$
, then $\mathcal{B}$ is a $Z$-tensor. The conclusion follows in the case.

Now we suppose that $\hat{J}(\mathcal{B})\neq \emptyset$, let $\mathcal{B}_1=\mathcal{B}=(b_{i_1\cdots i_m}^{(1)})$, and let  $d_i^{(1)}$ be be the value of the largest off-diagonal entry in the $i$th row of $\mathcal{B}_1$, that is, \[d_i^{(1)}=\max\limits_{i_2\ldots i_m \in N, \atop \delta_{ii_2\ldots i_m}=0} b_{ii_2\cdots i_m}^{(1)}.\]
Furthermore, let $\hat{J}_1 =\hat{J}(\mathcal{B}_1)$,
$h_1=\min\limits_{i\in \hat{J}_1} d_i^{(1)}$
and \[J_1=\{i\in \hat{J}_1:d_i^{(1)}=h_1\}.\]
Then $J_1 \subseteq \hat{J}_1$ and $h_1>0$.\

Consider $\mathcal{B}_2=\mathcal{B}_1-h_1 \varepsilon^{\hat{J}_1}=(b_{i_1\cdots i_m}^{(2)})$. Obviously, $\mathcal{B}_2$ is also symmetric by the definition of $\varepsilon^{\hat{J}_1}$. Note that
\begin{equation}\label{eq3.5}
b_{i_1\cdots i_m}^{(2)}=\left\{\begin{array}{cccc}
   b_{i_1\cdots i_m}^{(1)}-h_1,   &i_1,i_2,\ldots,i_m\in \hat{J}_1   \\
   b_{i_1\cdots i_m}^{(1)},   &otherwise,
\end{array}
\right.\end{equation}
for $i\in J_1$,
\begin{equation}\label{eq3.7}\beta_i(\mathcal{B}_2)=\beta_i(\mathcal{B}_1)-h_1=0,
\end{equation}
and that for $i\in \hat{J}_1\backslash J_1$,
\begin{equation}\label{eq3.6}\beta_i(\mathcal{B}_2)=\beta_i(\mathcal{B}_1)-h_1> 0.
\end{equation}
Combining (\ref{eq3.5}), (\ref{eq3.7}), (\ref{eq3.6}) with the fact that for each $j\notin \hat{J}_1$,
$\beta_i(\mathcal{B}_2)=\beta_i(\mathcal{B}_1)$, we easily obtain by Definition \ref{def03.1} that $\mathcal{B}_2$ is still a symmetric quasi-double $B$-tensor.

Now replace $\mathcal{B}_1$ by $\mathcal{B}_2$, and repeat this
process. Let $\hat{J}(\mathcal{B}_2) =\{i\in
N:$~there~is~at~least~one positive~off-diagonal~entry~in~the
\emph{i}th~row~of $\mathcal{B}_2\}$. Then
$\hat{J}(\mathcal{B}_2)=\hat{J}_1\backslash J_1$. Repeat this
process until $\hat{J}(\mathcal{B}_{s+1}) = \emptyset$. Let
$\mathcal{M}=\mathcal{B}_{s+1}$. Then  (\ref{eq300.4}) holds.

Furthermore, if $m$ is even, then $\mathcal{B}$ a symmetric quasi-double $B$-tensor of even order.  If $\mathcal{B}$ itself is a Q-DSDD symmetric $Z$-tensor, then it is positive definite by Lemma \ref{lem30.2}. Otherwise, (\ref{eq300.4}) holds with $s>0$. Let $x\in \mathbb{R}^{n}$. Then by (\ref{eq300.4}) and that fact that $\mathcal{M}$ is positive definite, we have
\[\mathcal{B}x^{m}=\mathcal{M}x^{m}+ \sum\limits_{k=1}^{s}h_k\varepsilon^{\hat{J}_k}x^{m}
= \mathcal{M}x^{m}+ \sum\limits_{k=1}^{s}h_k||x_{\hat{J}_k}||_m^{m}\geq  \mathcal{M}x^{m}> 0.\] This implies that $\mathcal{B}$ is positive definite. Note that a symmetric tensor is a $P$-tensor if and only it is positive definite \cite{So1}, therefore $\mathcal{B} $ is a $P$-tensor. The proof is completed. \end{proof}

Similar to the proof of Theorem \ref {th3.2}, by Lemma \ref{lem30.1} we easily have that an even order symmetric double $B$-tensor is  positive definite and a $P$-tensor.

\begin{thm} \label{th3.3}
Let $\mathcal{B}=(b_{i_1\cdots i_m})$ be a symmetric double $B$-tensor of order $m$ dimension $n$. Then either $\mathcal{B}$ is a DSDD symmetric $Z$-tensor itself, or we have
\begin{equation}\label{eq3.4}
\mathcal{B}=\mathcal{M}+ \sum\limits_{k=1}^{s}h_k\varepsilon^{\hat{J}_k},
\end {equation}
where $\mathcal{M}$ is a DSDD symmetric $Z$-tensor, $s$ is a positive integer, $h_k >0$ and $\hat{J}_k\subseteq N$, for $k=1,2,\cdots,s$. Furthermore, If $m$ is even, then $\mathcal{B}$ is
 positive definite, consequently, $\mathcal{B}$ is a $P$-tensor.
\end{thm}

Since an even order real symmetric tensor is positive definite if and only if all of its H-eigenvalues are positive \cite{Qi},  by Theorems \ref{th3.2} and \ref{th3.3}  we have the following results.

\begin{corollary}
All the H-eigenvalues of an even order symmetric double $B$-tensor are positive.
\end{corollary}

\begin{corollary}
All the H-eigenvalues of an even order symmetric quasi-double $B$-tensor are positive.
\end{corollary}

\section{Conclusions} In this paper, we give two generalizations of $B$-tensors:
double $B$-tensors and quasi-double $B$-tensors, and prove that an
even order symmetric (quasi-)double $B$-tensor is positive definite.

On the other hand, we could consider the problem that whether an
even order symmetric tensor is  positive semi-definite by weakening
the condition of Definition \ref{def03.1} as follows.

\begin{definition}\label{def4.2}
Let $\mathcal{B}=(b_{i_1\cdots i_m})$ be a tensor of order $m$ dimension $n$.
$\mathcal{B}$ is a quasi-double $B_0$-tensor if and only if for all $i,j\in N$ $i\neq j$,
\begin{equation} \label{eq4.2}
(b_{i\cdots i}-\beta_i(\mathcal{B})) \left (b_{j\cdots j}-\beta_j(\mathcal{B})- \Delta_j^i ({\mathcal{B}}) \right)
\geq \left( \beta_j(\mathcal{B})-b_{ji\cdots i}\right)\Delta_i ({\mathcal{B}}).
\end{equation}
\end{definition}

However, it con't be proved by using the technique in this paper
that an even order symmetric quasi-double $B_0$-tensor is positive
semi-definite. We here only give the following conjecture.

\begin{conjecture} An even order
symmetric quasi-double $B_0$-tensor is positive semi-definite.
\end{conjecture}
\section*{Acknowledgements}
The authors would like to thank Professor L. Qi for his many
valuable comments and suggestions. 







\end{document}